\documentclass[letter, 12pt]{article}
\usepackage{amssymb}
\usepackage{amsthm}
\usepackage{amsmath}
\usepackage{mathrsfs}
\usepackage{verbatim}
\newcommand{\argmax}{\operatornamewithlimits{argmax}}
\newcommand{\argmin}{\operatornamewithlimits{argmin}}
\title{Convexity and multi-dimensional screening for spaces with different dimensions \footnote{These results are based on the author's doctoral thesis in mathematics at the University of Toronto.  He is grateful to his thesis advisor, Robert McCann, for many stimulating discussions, and to Guillaume Carlier and Lars Stole for helpful comments regarding the uniqueness of optimal pricing plans.  The author was supported in part by an NSERC postgraduate scholarship and a OGSST scholarship.}}
\author{BRENDAN PASS \footnote{Department of Mathematical and Statistical Sciences, 632 CAB, University of Alberta, Edmonton, Alberta, Canada, T6G 2G1 pass@ualberta.ca.}}

\begin{document}

\maketitle

\begin{abstract}
We study the principal-agent problem.  We show that $b$-convexity of the space of products, a condition which appears in a recent paper by Figalli, Kim and McCann \cite{fkm}, is necessary to formulate the problem as a maximization over a convex set.  We then show that when the dimension $m$ of the space of types is larger than the dimension $n$ of the space of products, this condition implies that the extra dimensions do not encode independent economic information.  When $m$ is smaller than $n$, we show that under $b$-convexity of the space of products, it is always optimal for the principal to offer goods only from a certain prescribed subset.  We show that this is equivalent to offering an $m$-dimensional space of goods. 
\end{abstract}

\section{Introduction}
This paper concerns principal-agent type problems, which arise frequently in a variety of different contexts in economic theory.  The following formulation can be found in Wilson \cite{wil}, Armstrong \cite{arm} and Rochet and Chone \cite{rc}.  A monopolist wants to sell goods to a distribution of buyers.  Knowing only the preference $b(x,y)$ that a buyer of type $x \in X$ has for a good of type $y \in Y$, the density $d\mu(x)$ of the buyer types and the cost $c(y)$ to produce the good $y$, the monopolist must decide which goods to produce and how much to charge for them in order to maximize her profits.

When the distribution of buyer types $X$ and the available goods $Y$ are either discrete or $1$-dimensional, this problem is well understood \cite{spen}\cite{mir}\cite{mr}\cite{barmy}.  However, it is typically more realistic to distinguish between both consumer types and goods by more than one characteristic.  An illuminating illustration of this is outlined by Figalli, Kim and McCann \cite{fkm}: consumers buying automobiles may differ by, for instance, their income and the length of their daily commute, while the vehicles themselves may vary according to their fuel efficiency, safety, comfort and engine power, for example.  It is desirable, then, to study models where the respective dimensions $m$ and $n$ of $X$ and $Y$ are greater than $1$ \cite{McAMcM}\cite{QR}\cite{rs}.  This \textit{multi-dimensional} screening problem is much more difficult and relatively little is known about it; for a review and an extensive list of references, see the book by Basov \cite{bas2} 

When $m=n$ and the preference function $b(x,y):=x \cdot f(y)$ is linear in types, Rochet and Chone developed an algorithm for studying this problem \cite{rc}.  A key element in their analysis is that, in this case, the problem may be formulated mathematically as an optimization problem over the set of convex functions, which is itself a convex set. They were then able to deduce the existence and uniqueness of an optimal pricing strategy, as well as several interesting economic characteristics of it.  Basov then analyzed the case where $b$ is linear in types but $m \neq n$ \cite{bas}.  When $m<n$, he was able to essentially reduce the $n$-dimensional space $Y$ to an $m$-dimensional space of \textit{artificial} goods and then apply the machinery of Rochet and Chone.  When $m>n$, no such reduction is possible in general.  Under additional hypotheses, however, he showed that the solution actually coincides with the solution to a similar problem where both spaces are $m$-dimensional.

For more general preference functions, Carlier, using tools from the theory of optimal transportation, was able to formulate the problem as the maximization of a functional $P$ over a certain set of functions $U_{b,\phi}$ (a subset of the so called $b$-convex functions, which will be defined below) \cite{car}.  He was then able to assert the existence of a solution to this problem; that is, the existence of an optimal pricing schedule; an equivalent result is also proved in \cite{mp}.  However, for general functions $b$, the set of $b$-convex functions may not be convex and so characterizing the solution using either computational or theoretical tools is an extremely imposing task.  Very little progress had been made in this direction until recently, when Figalli, Kim and McCann \cite{fkm} found necessary and sufficient conditions on $b$ for $U_{b,\phi}$ to be convex, assuming $n=m$.  Convexity is a very powerful tool in variational problems of this type, which can be exploited both theoretically and numerically.  Assuming in addition that the cost $c$ is $b$-convex, Figalli, Kim and McCann demonstrated that the functional $P$ is concave and from here were able to prove uniqueness of the solution and demonstrate that some of the interesting economic features observed by Rochet and Chone persist in this setting.  Surprisingly, the tools they use are also adapted from an optimal transportation context; their necessary and sufficient condition is derived from a condition developed by Ma, Trudinger and Wang \cite{mtw}, governing the regularity of optimal maps.

Although the result of Figalli, Kim and McCann represents major progress, it is limited in that they had to assume that the spaces of types and products were of the same dimension.  There are many interesting and relevant economic models in which these spaces have different dimensions, as outlined in, for example, Basov \cite{bas2}.  Our primary goal here is to study how the results in \cite{fkm} extend to the case when $m \neq n$; in particular, we want to determine under what conditions the set of $b$-convex functions is convex for general values of $m$ and $n$.  Our first contribution is to establish a necessary condition for the convexity of this set.  This condition, known as $b$-convexity of $Y$, was a hypothesis in \cite{fkm}; prior to that, to the best of my knowledge, it had not been explored in the principal-agent context, although it is well known in the optimal transportation literature, since the work of Ma, Trudinger and Wang \cite{mtw}.

We then study separately the cases $m>n$ and $m<n$.  The analysis here parallels the author's recent work on the regularity of optimal transportation between spaces whose dimensions differ \cite{p}.  When $m>n$, we show that the $b$-convexity of $Y$ implies that the dimensions cannot differ in a meaningful way.  That is, although $b$ may \emph{appear} to depend on an $m$ dimensional variable, there is a natural disintegration of $X$ into smooth sub-manifolds of dimension $m-n$ such that, no matter how the monopolist sets her prices, types in the same sub-manifold \textit{always} choose the same good.  Therefore, types in the same set are indistinguishable, and rather than working in an $m$ dimensional space, we may as well identify the types in a single sub-manifold and work instead in the resulting $n$ dimensional quotient space. 

When $n>m$, the generalized Spence-Mirrlees single crossing condition, found for example in Basov \cite{bas2}, cannot hold and consequently consumers' marginal utilities cannot uniquely determine which product they buy.  In this case, given a price schedule, a certain buyer's surplus may be maximized by many different goods, making him indifferent between those goods.  The monopolist's profits will be very different, however, depending on which good the buyer chooses.  A naive possible solution would be to only produce from the indifference set the good which maximizes the monopolist's profit; however, in doing so, she may exclude a good which would maximize her profit from another buyer.  It turns out that the $b$-convexity on $X$ (which was also an assumption  in \cite{fkm}) precludes this from happening; under this condition, we can again reduce the problem to one where the two spaces share the same dimension.  A special case of this result where $b(x,y) =x \cdot v(y)$ for a function $v: Y \mapsto  \mathbb{R}^m$ was established by Basov \cite{bas}.

\section{Assumptions and mathematical formulation}
We will assume that the space of types $X \subseteq \mathbb{R}^m$ and the space of goods $Y \subseteq \mathbb{R}^n$ are open and bounded.

Before formulating the problem mathematically, we recall the conditions on $b$ imposed by Figalli, Kim and McCann.  Our formulations will appear slightly different, as they must apply to the more general case $m \neq n$; when $n=m$ they coincide exactly with the conditions in \cite{fkm}.  In what follows, $D_xb(x,y) \in \mathbb{R}^m$ and $D_yb(x,y) \in \mathbb{R}^n$ will denote the differentials of $b$ with respect to $x$ and $y$ respectively.  $D^{2}_{xy}b(x,y)$ will denote the $m \times n$ matrix of mixed, second order, partial derivatives.\\
\\
\textbf{(B0)}: The function $b \in C^4(\overline{X} \times \overline{Y})$.\\
\textbf{(B1)}: (bi-twist) For all $x_0 \in \overline{X}$ and $y_0 \in \overline{Y}$, the level sets of the maps $y \mapsto D_xb(x_0,y) $ and $x \mapsto D_yb(x,y_0)$ are connected and $D^{2}_{xy}b(x_0,y_0)$ has full rank.\\
\textbf{(B2)}: For all $x_0 \in \overline{X}$ and $y_0 \in \overline{Y}$, the images $D_xb(x_0, \overline{Y})$ and $D_yb(\overline{X}, y_0)$ are convex.  If $D_xb(x_0, \overline{Y})$ is convex for all $x_0$, we say that $\overline{Y}$ is $b$-convex, while if $D_yb(\overline{X}, y_0)$ is convex for all $y_0$ we say that $\overline{X}$ is $b$-convex.\\
\textbf{(B3)}: For all $x_0 \in X$ and $y_0 \in Y$, we have 
\begin{equation*}
\frac{\partial^4}{\partial t^2 \partial s^2}b(x(s), y(t))\Big|_{(s,t)=(0,0)}\geq 0,
\end{equation*} 
whenever the curves $s \in [-1,1] \mapsto D_yb(x(s), y_0)$ and $t \in [-1,1] \mapsto D_xb(x_0, y(t))$ form affinely parametrized line segments.\\
\textbf{(B3u)}: \textbf{(B3)} holds and, whenever $\dot{x}(0)\cdot D^{2}_{xy}b(x_0,y_0) \neq 0$ and $D^{2}_{xy}b(x_0,y_0)\cdot \dot{y}(0) \neq0$, the inequality is strict. \\
\\
As was emphasized by Figalli, Kim and McCann, these conditions are invariant under reparametrization of $X$ and $Y$.  This means that they are in some sense economically natural;  they do not depend on the coordinates used to parametrize the problem \cite{fkm}.

Let us take a moment to explain the meaning of condition \textbf{(B1)}.  Assume momentarily that $m \geq n$.  Then the full rank condition implies that $y \mapsto D_xc(x_0,y)$ is locally injective and so connectedness of its level sets implies its global injectivity.  Hence, we recover the generalized Spence-Mirrlees, or generalized single crossing, condition found in, for example, Basov \cite{bas2} (more precisely, we obtain the strengthened version in \cite{fkm}).  On the other hand, if $m<n$, the generalized Spence-Mirrlees condition cannot hold; however, as we will establish, in certain cases $\textbf{(B1)}$ is a suitable replacement.

Much of our attention here will be devoted to \textbf{(B2)}.  For a bilinear $b$, this condition coincides with the usual notion of convexity of the sets $\overline{X}$ and $\overline{Y}$; for more general $b$, it implies convexity of $\overline{X}$ and $\overline{Y}$ after an appropriate change of coordinates \cite{fkm}.  We will see in the next section that the convexity of $D_xb(x_0, \overline{Y}) \subseteq \mathbb{R}^m$ is a necessary condition for the monopolist's problem to be a convex program; in section 5, we will show that when $m<n$ the convexity of $D_yb(\overline{X}, y_0)\subseteq \mathbb{R}^n$ reduces the problem to a more tractable problem in equal dimensions.  

The relevance of \textbf{(B3)} and \textbf{(B3u)} to economic problems was established in \cite{fkm}.  They are, respectively, strengthenings of the conditions \textbf{(A3w)} and \textbf{(A3)}, which are well known in optimal transportation due to their intimate connection with the regularity of optimal maps \cite{mtw} \cite{loeper}.  Background on optimal transportation can be found in \cite{V} \cite{V2}.

We are now ready to review the mathematical formulation of the principal-agent problem.  Suppose that the monopolist sets a price schedule $v(y)$; $v(y)$ is the price she charges for good $y$.  Buyer $x$ chooses to buy the good that maximizes $b(x,y)-v(y)$.  We therefore define the utility for buyer $x$ to be 

\begin{equation*}
 v^b(x)= \sup_{y \in Y} b(x,y)-v(y)
\end{equation*}

Functions of this type are called $b$-convex functions; we will denote by $U_b$ the set of all such functions.

We assume the existence of a $y_{\phi} \in \overline{Y}$ that the monopolist \textit{must} offer at cost; that is, for any price schedule $v$
\begin{equation}
 v(y_{\phi})=c(y_{\phi}) \label{null}
\end{equation}

If both sides in equation (\ref{null}) are equal to zero, we can interpret $y_{\phi}$ as the null good, and equation  (\ref{null}) as representing the consumers' option not to purchase any product (and the monopolist's obligation not to charge them should they exercise this option).  Note that the restriction $v(y_{\phi})=c(y_{\phi})$ immediately implies $v^b(x) \geq u_{\phi}(x):=b(x,y_{\phi})-c(y_{\phi})$.

Let  $y_{v^b}(x) \in \argmax_{y \in \overline{Y}}(b(x,y)-v(y))$.   Assuming that a buyer of type $x$ chooses to buy good $y_{v^b}(x)$ \footnote{The generalized Spence-Mirlees condition implies that for almost all $x$, there is exactly one $y$ maximizing $b(x,y)-v(y)$, and so under this condition, the function $y_{v^b}$ is uniquely determined from $v^b$ almost everywhere.}, the monopolist's profits from this buyer is then $v(y_{v^b}(x))-c(y_{v^b}(x))=b(x,y_{v^b}(x))-v^b(x)-c(y_{v^b}(x))$ and her total profits are:

\begin{equation*}
 P(v^b) :=\int_{x}b(x,y_{v^b}(x))-v^b(x)-c(y_{v^b}(x))d\mu(x)
\end{equation*}

The monopolist's goal, of course, is to maximize her profits.  That is, to maximize $P(v^b)$ over the set $U_{b,\phi}$ of $b$-convex functions which are everywhere greater than $u_{\phi}$ (and, if the generalized Spence-Mirlees condition fails to hold, over all functions $y_{v^b}(x) \in \argmax_{y \in \overline{Y}}(b(x,y)-v(y))$).

The main result of \cite{fkm} is that when $m=n$, under hypotheses \textbf{(B0)}-\textbf{(B2)} convexity of the set $U_{b,\phi}$ is equivalent to \textbf{(B3)}.

\section{b-convexity of the space of products}
This section establishes the following result, which is novel even when $m=n$.
\newtheorem{connec}{Proposition}[section]
\begin{connec}\label{connec}
 If $\overline{Y}$ is not $b$-convex at some point $x \in X$, the set $U_{b,\phi}$ is not convex.
\end{connec}
\begin{proof}
  Suppose $Y$ is not $b$-convex at $x \in X$.  Then there exist $y_0, y_1 \in \overline{Y}$ and a $t \in (0,1)$ such that $(1-t)\cdot D_xb(x, y_0) + t \cdot D_xb(x,y_1) \notin D_xb(x,Y)$.

Now, choose $b$-convex functions $v^{b}_0, v^{b}_1 \geq u_{\phi}$ such that $v^{b}_i$ is differentiable at $x$ and $Dv^{b}_i(x)=D_xb(x, y_i)$, for $i=0,1$.  Define $v^{b}_t=(1-t)\cdot v^{b}_0 +t\cdot v^{b}_1$; we will show that $v^{b}_t$ is not $b$-convex, which will imply that $U_{b,\phi}$ is not convex.  Now,
\begin{eqnarray}\nonumber
Dv^{b}_t(x) &=& (1-t)\cdot Dv^{b}_0(x) +t\cdot Dv^{b}_1(x)\\ 
&=&(1-t)\cdot D_xb(x, y_0) + t \cdot D_xb(x,y_1) \notin D_xb(x,\overline{Y}) \label{notcon}
\end{eqnarray}

Now, assume $v^{b}_t$ is $b$-convex; then 
\begin{equation}
v^{b}_t(x) =\sup_{y \in Y}b(x,y)-v_t(y)\label{vb} 
\end{equation}
for some price schedule $v_t$. Without loss of generality, we may assume $v_t$ is $b$-convex: $v_t(x) =\sup_{x \in X}b(x,y)-v^b_t(x)$, which implies that $v_t(x)$ is continuous \cite{gm}.  By compactness of $\overline{Y}$ and continuity of $y \mapsto b(x,y)-v_t(y)$, the supremum in (\ref{vb}) is attained by some $y_t \in \overline{Y}$, $v^{b}_t(x) =b(x,y_t)-v^b(y_t)$.  Now, for all $\overline{x} \in X$, we have $v^{b}_t(\overline{x}) \geq b(\overline{x},y_t)-v_t(y_t)$ and so the function $\overline{x} \mapsto v^{b}_t(\overline{x})-b(\overline{x},y_t)$ is minimized at $\overline{x}=x$.  It now follows that $Dv^{b}_t(x)=D_xb(x,y_t) \in D_xb(x,Y)$, contradicting (\ref{notcon}).  We conclude that $v^{b}_t$ cannot be $b$-convex.  As $v^{b}_t$ is a weighted average of $b$-convex functions, this yields the desired result.
\end{proof}

\newtheorem{stren}[connec]{Remark}
\begin{stren}
 This result can be seen as a slight strengthening of the result of Figalli, Kim and McCann \cite{fkm}; assuming $n=m$, \textbf{(B0)}, \textbf{(B1)} and the $b$-convexity of $\overline{X}$, the main result of \cite{fkm} combines with Proposition \ref{connec} to imply that the convexity of $U_{b,\phi}$ is equivalent to the $b$-convexity of $\overline{Y}$ and \textbf{(B3)}.  We will see in the next section that this extends nominally to the case $m>n$, although, as we will show, in that case $\overline{Y}$ cannot be $b$-convex unless all the economic information encoded in $X$ can actually be encoded in an $n$-dimensional space. 
\end{stren}

An important consequence on the convexity of $U_{b,\phi}$ is the uniqueness of the optimal pricing schedule.  Assuming that the cost $c$ is strictly $b$-convex, $P$ is strictly concave functionl; if it is defined on a concave set, its maximum must be unique.  The following elementary example shows that when $b$-convexity of $Y$ (and hence convexity of $U_{b,\phi}$) fails, the principal's optimal strategy may not be unique.
\newtheorem{nonunique}[connec]{Example}
\begin{nonunique}
Let $X=[0,1]$ be the unit interval and $Y=\{0,1\}$ be a set of two points, including the null good $0$; that is, the principle only offers one good, $y=1$.  Take $b(x,y)=xy+y$ to be bilinear and $c(y)=y^2$.  Let the density of consumer types be $f(x)=60x^2-80x+29$.  To make a profit, the price $v$ the principal sets for her good must be between $1$ and $2$; she clearly cannot make money by charging less than the cost $1^2=1$ of producing the good $y=1$, and if she sets the price higher than $2$, every consumer would opt out.  A straightforward calculation shows that her profits are $(v-\frac{3}{2})^2-20(v-\frac{3}{2})^4+1$ which is maximized at $v=\frac{3}{2}\pm \frac{1}{2\sqrt{10}}$.

The profit functional, written in terms of the utility functions $v^b(x)=\sup_{y \in Y}xy+y-v(y)$ is 
\begin{equation*}
\int_0^1 x\frac{dv^b}{dx}+\frac{dv^b}{dx}-v^b-(\frac{dv^b}{dx})^2 dx
\end{equation*}
which is strictly concave.  However, the only allowable price schedules are of the form $v(0)=0, v(1) = v \in [1,2]$ and so the only allowable utility functions are of the form 
\begin{eqnarray}
v^b(x)&=&\max _{y =0,1}xy+y-v(y) \\
&=&\max\{0, x-v+1\} 
\end{eqnarray}
for some constant $v \in [1,2]$.  The convex interpolant of two functions of this form fails to have the same form; that is, the set of allowable utilities is not convex, precisely because the space $Y$ is not convex (recall that convexity and $b$-convexity of $Y$ are equivalent for bilinear preferences).  Hence, uniqueness fails.  If the principal had access to a convex set of goods (for example, the whole space $[0,1]$) she could construct a more sophisticated pricing strategy which would earn her a higher profit than either of the maxima exhibited in this example.

\end{nonunique}
\section{$m>n$}

In this section we focus on the case where $m>n$.  We will show that the $b$-convexity of the space of products $Y$ implies that $X$ can be reduced to an $n$-dimensional space without losing any economic information.  The analysis in this section strongly parallels the author's work on the regularity of optimal maps in an analogous setting \cite{p}.  Though many of the proofs in this section are similar to those in \cite{p}, we reproduce them here for the reader's convenience.

The definition below concerns the subset $D_xb(x,Y) \subseteq \mathbb{R}^m$; in general, condition \textbf{(B1)} ensures that this set is an $n$-dimensional submanifold of $\mathbb{R}^m$.
\newtheorem{lin}{Definition}[section]
\begin{lin}
We say the domain $Y$ looks $b$-linear from $x \in X$ if $D_xb(x,Y)$ is contained in a shifted $n$-dimensional, linear subspace of $\mathbb{R}^m$.  We say $Y$ is $b$-linear with respect to $X$ if it looks $b$-linear from every $x \in X$.
\end{lin}

When $m=n$, $b$-linearity is automatically satisfied.  When $m>n$, this is no longer true; $D_xb(x,Y)$ is an $n$-dimensional submanifold that may or may not be contained in an $n$-dimensional shifted linear subspace.  However, if $D_xb(x,Y)$ is convex, it must be contained in such a subspace and so $b$-convexity clearly implies $b$-linearity.  

We will also have reason to consider the level set of $\overline{x} \mapsto D_yb(\overline{x},y)$ passing through $x$, $L_x(y):=\{\overline{x} \in X: D_yb(\overline{x},y)=D_yb(x,y)\}$. 

The relationship between $b$-linearity and the sets $L_x(y)$ is expressed by the following result.

\newtheorem{linlev}[lin]{Lemma}
\begin{linlev}\label{linlev}(i) $Y$ looks $b$-linear from $x \in X$ if and only if $T_x(L_x(y))$ is independent of $y$; that is $T_x(L_x(y_0))=T_x(L_x(y_1))$ for all $y_0,y_1 \in Y$.
\\(ii) If the level sets $L_x(y)$ are all connected, then $Y$ is $b$-linear with respect to $X$ if and only if $L_x(y)$ is independent of $y$ for all $x$
\end{linlev}

\begin{proof}
We first prove (i).  The tangent space $T_x(L_x(y))$ to $L_x(y)$ at $x$ is the null space of the matrix $D^2_{yx}b(x,y)$, which, in turn, is the orthogonal complement of the range of its transpose, $D^2_{xy}b(x,y)$.  Therefore, $T_x(L_x(y))$ is independent of $y$ if and only if the range of $D^2_{xy}b(x,y)$ is independent of $y$.   But $D^2_{xy}b(x,y)$ is the derivative of the map $y \mapsto D_xb(x,y)$, and so its range is independent of $y$ if and only if the image of this map is linear.

To see (ii), note that (i) implies $Y$ is $b$-linear with respect to $X$ if and only if $T_x(L_x(y_0))=T_x(L_x(y_1))$ for all $x \in X$  and all $y_0,y_1 \in Y$.  But $T_x(L_x(y_0))=T_x(L_x(y_1))$ for all $x$ is equivalent to $L_x(y_0)=L_x(y_1)$ for all $x$; this immediately yields (ii).
\end{proof}

Our goal is to identify conditions on $b$ under which $U_{b,\phi}$ is convex; Proposition \ref{connec} and Lemma \ref{linlev} imply that this cannot be the case unless the sets $L_x(y)$ are independent of $y$.  For the rest of this section, we will therefore assume that the sets $L_x(y)$ are in fact independent of $y$; we will henceforth denote them simply by $L_x$.  We will show next that, no matter what pricing schedule the principal chooses, consumers in the same $L_x$ will always choose the same good and so, at least for the purposes of this problem, different points in the same $L_x$ do not really represent different types.

\newtheorem{ccon}[lin]{Lemma}
\begin{ccon} \label{con}
Assume each $L_x(y)$ is connected and independent of $y$.  For any $x_0, x_1 \in X$ such that $x_0 \in L_{x_1}$, $\overline{y} \in Y$ and $b$-concave $v^b$ we have $v^b(x_0)=b(x_0, \overline{y})-v(\overline{y})$ if and only if $v(x_1)=b(x_1, \overline{y})-v(\overline{y})$. 
\end{ccon}

\begin{proof}
 First note that as $D_yb(x_0,y)-D_yb(x_1,y)=0$ for all $y \in Y$, the difference $b(x_0,y)-b(x_1,y)$ is independent of $y$.  Now, suppose $v^b(x_0)=b(x_0, \overline{y})-v(\overline{y})$.  Then 
\begin{eqnarray*}
 v^b(x_1)&=&\inf_{y \in Y} b(x_1, y)-v(y) \\
&=&\inf_{y \in Y} \big( b(x_1, y)-b(x_0,y)+b(x_0,y)-v(y) \big)\\
&=& b(x_1, \overline{y})-b(x_0,\overline{y})+\inf_{y \in Y} \big( b(x_0,y)-v(y) \big)\\
&=& b(x_1, \overline{y})-b(x_0,\overline{y})+v^b(x_0) \\ 
&=& b(x_1, \overline{y})-v(\overline{y}) \\ 
\end{eqnarray*}
The proof of the converse is identical.
\end{proof}

We can now reformulate the monopolist's problem as a problem between two $n$-dimensional spaces.  To do this, we define an \textit{effective} space of types, by essentially identifying all consumer types in a single $L_x$ as a single \textit{effective} type.

Fix some $y_0 \in Y$ and define the space of \emph{effective} types $Z:=D_yb(X,y_0) \subseteq \mathbb{R}^n$ and the map $Q:X \rightarrow Z$ via $Q(x) := D_yb(x,y_0)$.  We define an \emph{effective} preference function: $h: Z \times Y \rightarrow \mathbb{R}$ via

\begin{eqnarray*}
h(z,y)=b(x,y)-b(x,y_0), 
\end{eqnarray*}
where $x \in Q^{-1}(z)$.  We must check that $h$ is well defined, that is 
\begin{equation*}
 b(x,y)-b(x,y_0)=b(\overline{x},y)-b(\overline{x},y_0), 
\end{equation*}
or equivalently 
\begin{equation*}
 B(x,\overline{x}, y, y_0):=b(x,y)-b(x,y_0)-b(\overline{x},y)+b(\overline{x},y_0)=0, 
\end{equation*}
for all $\overline{x} \in L_x$ and $y \in Y$.  This is easily verified; the identity clearly holds at $y=y_0$ and as $D_yB(x,\overline{x}, y, y_0)$ vanishes, it must hold for all $y$.

Given a price schedule $v(y)$, we define the corresponding \emph{effective} utility as,
\begin{eqnarray}
 v^h(z) & = & \sup_{y \in Y}h(z,y)-v(y) \nonumber\\
 & = &  \sup_{y \in Y}b(x,y)-b(x,y_0)-v(y)\nonumber\\
 & = & -b(x,y_0)+ \sup_{y \in Y}b(x,y)-v(y)\nonumber\\
& = & -b(x,y_0)+ v^b(x) \label{util}
\end{eqnarray}
for any $x \in Q^{-1}(z)$.  An effective consumer of type $z$ chooses the product at which this supremum is attained; we define this map to be $y_{v^h}(z)$.  It is clear from the preceding calculation that, for every $x \in Q^{-1}(z)$ we have $y_{v^b}(x)=y_{v^h}(z)$.  Define the distribution of effective consumer types to be the pushforward measure $\nu=Q_{\#}\mu$ .  Define the monopolist's effective profits to be 
\begin{equation*}
 P_{eff}(v^h)=\int_Z(h(z,y_{v^h}(z))-v^h(z)-c(y_{v^h}(z)))d\nu(z)
\end{equation*}

The next theorem implies that maximizing the monopolist's effective profits is equivalent to maximizing her profits. 

\newtheorem{eff}[lin]{Theorem}
\begin{eff}\label{eff}
For any pricing schedule, the monopolist's effective profits are equal to her profits.  
\end{eff}

\begin{proof}
\begin{eqnarray*}
 P(v^b)&=&\int_X(b(x,y_{v^b}(x))-v^b(x)-c(y_{v^b}(x)))d\mu(x)\\
&=& \int_X(b(x,y_{v^h}(Q(x)))-v^h(Q(x))-b(x,y_0)-c(y_{v^h}(Q(x))))d\mu(x)\\
&=&\int_X(h(Q(x),y_{v^h}(Q(x)))-v^h(Q(x))-c(y_{v^h}(Q(x))))d\mu(x)\\
&=&\int_Z(h(z,y_{v^h}(z))-v^h(z)-c(y_{v^h}(z)))d\nu(x)\\
&=& P_{eff}(v^h)\\
\end{eqnarray*}
\end{proof}

The following corollary now follows immediately.
\newtheorem{class}[lin]{Corollary}
\begin{class}  Assume $\textbf{(B0)}$ and  $\textbf{(B1)}$.  Then one of the following holds.
\begin{enumerate}
\item $U_{b, \phi}$ is not convex.
\item The space of types can be reduced to an $n$-dimensional space of effective types, and solving the monopolist's problem with this effective space of types is equivalent to solving her original problem.
\end{enumerate}
\end{class}

According to Figalli, Kim and McCann \cite{fkm}, this new, equal dimensional problem is a maximization over a convex set, provided that the conditions \textbf{(B0)}-\textbf{(B3)} hold for $h$, $Z$ and $Y$; meanwhile, to ensure strict convexity of the functional $P_{eff}$ and uniqueness of the optimizer, one needs either \textbf{(B3u)} and the $h$-convexity of $g$ or the strict $h$-convexity of $g$.  It is therefore desirable to be able to test these properties using only the information present in the original problem; that is, using $b$ and $X$ rather than $h$ and $Z$.

 Before proceeding, we make a useful (local) identification between $Z$ and certain subsets of $X$, following the analysis in \cite{p}.  Pick a point $z_0 \in Z$ and select $x_0 \in Q^{-1}(z_0)$.  Now, let $S$ be an $n$-dimensional surface passing though $x_0$ which intersects $L_{x_0}$ transversely; note that this implies $\mathbb{R}^m=T_xL_{x_0} \bigoplus T_xS$.  As the null space of the matrix $D^2_{yx}b(x_0,y_0)$ is precisely  $T_xL_{x_0}$ for any $y$, it is invertible when restricted to $T_{x_0}S$; by the inverse function theorem, the map $Q(\cdot)=D_yb(\cdot,y_0)$ restricts to a local diffeomorphism on $S$.  For all $z$ near $z_0$, there is a unique $x \in S \cap Q^{-1}(z)$ and we have $h(z,y)=b(x,y)-b(x,y_0)$; we can now identify $D_zh(z,y) \approx D_xb|_{S \times Y}(x,y)-D_xb|_{S \times Y}(x,y_0)$ and $D_{zy}^2h(z,y) \approx D_{xy}^2b|_{S \times Y}(x,y)$.  

We are now ready to prove the following theorem.

\newtheorem{prop}[lin]{Theorem}
\begin{prop}\label{prop}
(i) If $b$ satisfies \textbf{(B1)} on $X \times Y$, then $h$ satisfies \textbf{(B1)} on $Z \times Y$.\\
(ii) If $b$ satisfies \textbf{(B2)} on $X \times Y$, then $h$ satisfies \textbf{(B2)} on $Z \times Y$.
\end{prop}

\begin{proof}
 For (i), note that the hypotheses about the connectivity of level sets is equivalent to injectivity of the relevant maps for the equal dimensional spaces $Z$ and $Y$.  The injectivity of $z \mapsto D_yh(z,y)$ is an immediate consequence of the definition of $h$, while the injectivity of $y \mapsto D_zh(z,y)$ follows from the injectivity of $y \mapsto D_xb(x,y)$ and the identification above.  The full rank of the matrix $D^2_{zy}h$ follows easily from our identification, as $D^2_{xy}b$ has full rank when restricted to a transversal sub-manifold.

Half of (ii) is immediate, as $D_yh(Z,y)=D_yb(X,y)$.  To obtain the convexity of $D_zh(z,Y)$, we use the identification above; note that $D_zh(z,Y)$ is the projection of $D_xb(x,Y)$ onto the tangent space of $S$ and the projection of a convex set is convex.
\end{proof}

\newtheorem{b3}[lin]{Theorem}

\begin{b3}\label{b3} Assume that $b$ satisfies \textbf{(B0)}-\textbf{(B2)}.  Then\\
(i) $b$ satisfies \textbf{(B3)} on $X \times Y$ if and only if $h$ satisfies \textbf{(B3)} on $Z \times Y$.\\ 
(ii) $b$ satisfies \textbf{(B3u)} on $X \times Y$ if and only if $h$ satisfies \textbf{(B3u)} on $Z \times Y$. 
\end{b3}
\begin{proof}
We first prove (i).  Assuming the \textbf{(B3)} condition on $b$ , we have that $\frac{\partial^4}{\partial t^2 \partial s^2}b(x(s), y(t))\Big|_{(s,t)=(0,0)}\geq 0$ for appropriately chosen paths $x(s)$ and $y(t)$, assuming $x(s) \in S$.  Using the identification above, this implies that $h$ satisfies \textbf{(B3)}.  

Conversely, suppose now that $h$ satisfies \textbf{(B3)}.  Choose curves $x(s)$ in $X$ and $y(t)$ in $Y$ as in the definition of \textbf{(B3)}.  If $\dot{x}(0) \notin T_xL_{x(0)}$, then we can choose some surface $n$-dimensional surface $S$, intersecting $L_{x(0)}$ transversely, such that $x(s) \in S$ for small $s$.  In this case, using the local identification above, \textbf{(B3)} on $h$ implies

\begin{equation*}
\frac{\partial^4}{\partial t^2 \partial s^2}b(x(s), y(t))\Big|_{(s,t)=(0,0)}\geq 0,
\end{equation*} 

On the otherhand, suppose $\dot{x}(0) \in T_xL_{x(0)}$.  It is well known (see \cite{loeper} \cite{km} and \cite{fkm}) that as long as $D_xb(x(0), y(t))$ is an afinely parameterized line segment, we have 

\begin{equation}\label{newcurve} 
\frac{\partial^4}{\partial t^2 \partial s^2}b(x(s), y(t))\Big|_{(s,t)=(0,0)} = \frac{\partial^4}{\partial t^2 \partial s^2}b(\tilde{x}(s), y(t))
\end{equation}
for any other curve $\tilde{x}(s)$ in $X$ satisfying $\tilde{x}(0) = x(0)$ and $\dot{\tilde{x}}(0) = \dot{x}(0)$.  In particular, we may choose a curve $\tilde{x}(s)$ in $L_{x(0)}$ with these properties.  We will show that

\begin{equation*}
\frac{\partial^4}{\partial t^2 \partial s^2}b(\tilde{x}(s), y(t))\Big|_{(s,t)=(0,0)} = 0
\end{equation*}   
which, by (\ref{newcurve}), will imply the desired result.

Now, for all $s$ and $t$ we have 

\begin{equation*}
\frac{d \tilde{x}}{ds}(s) \in T_{\tilde{x}(s)}L_{\tilde{x}(s)}=\text{null}\big(D^2_{xy}b(\tilde{x}(s),y(t))\big)
\end{equation*}
and so 
\begin{equation*} 
  \frac{d^2}{dsdt}b(\tilde{x}(s),y(t))=\Big(\frac{d\tilde{x}(s)}{ds}\Big)^T\cdot D^2_{xy}b\big(\tilde{x}(s),y(t))\big) \cdot \frac{dy(t)}{dt}=0
 \end{equation*}
As this holds for all $s$ and $t$, we clearly have 

\begin{equation*}
\frac{\partial^4}{\partial t^2 \partial s^2}b(\tilde{x}(s), y(t))\Big|_{(s,t)=(0,0)} = 0
\end{equation*}

Turning now to (ii), if $b$ satisfies \textbf{(B3u)}, then so does $h$, using the identification exactly as in the argument for \textbf{(B3)}.  Converseley, suppose $h$ satisfies \textbf{(B3u)}.  For any curves $x(s)$ and $y(t)$ such that $\dot{x}(0)^T\cdot D^{2}_{xy}b(x_0,y_0) \neq 0$, we must have $\dot{x}(0) \notin T_xL_x$, as  $T_{x(0)}L_{x(0)}=\text{null}\big(D^2_{xy}c(x(s),y(0))\big)$.  Then, preceding as above, we can find some surface $S$, transversal to $L_{x(0)}$ such that $x(s) \in S$ for small $s$ and use the local identification to verify \textbf{(B3u)}.

\end{proof}
Theorems \ref{prop} and \ref{b3}, together with the main result in \cite{fkm} immediately imply the following result.  It generalizes the main result of \cite{fkm} to the $m \geq n$ setting, although, as we have shown in this section, when $m >n$ the hypothesis \textbf{(B2)} essentially reduces the problem to the $m=n$ case.

\newtheorem{b3con}[lin]{Corollary}
\begin{b3con}\label{b3con}
Assume $b$ satisfies \textbf{(B0)}-\textbf{(B2)}.  The $U_{b,\phi}$ is convex if and only if $b$ satisifes \textbf{(B3)}. 
\end{b3con}

Figalli, Kim and McCann also proved the concavity of the functional $P$ provided that the cost function $c$ is $b$ convex \cite{fkm}.   In addition, recall that a $b$-convex function $c$ is called strictly $b$-convex if the function $y_{c^b}(x)=\argmax_{y \in \overline{Y}} b(x,y)-v(y)$ is continuous.  Figalli, Kim and McCann showed that $P$ is strictly concave (a sufficient condition for the uniqueness of the optimal pricing strategy) provided either (i)\textbf{(B3)} holds and $c$ is strictly $b$-convex or (ii) \textbf{(B3u)} holds.  We show below that the (strict) $b$-convexity of $c$ is equivalent to its (strict) $h$-convexity.

\newtheorem{screenbcon}[lin]{Proposition}
\begin{screenbcon}\label{screencon}
Assume $b$ satisfies \textbf{(B0)} - \textbf{(B2)}.  Then \\
(i)$c$ is $b$-convex if and only if it is $h$-convex.  \\
(ii)$c$ is strictly  $b$-convex if and only if it is strictly $h$-convex.
\end{screenbcon}
\begin{proof}
 $b$-convexity of $c$ means that $c(y) = \sup_{x \in X} b(x,y)-c^b(x)$; it's $h$ convexity means that $c(y) =\sup_{z \in Z}h(z,y)-c^h(z)$.  Therefore, to prove (i), it suffices to show $\sup_{x \in X} b(x,y)-c^b(x) = \sup_{z \in Z}h(z,y)-c^h(z)$.  we have
\begin{eqnarray*}
\sup_{x \in X} b(x,y)-c^b(x)&=&\sup_{x \in X}b(x,y)-b(x,y_0)+b(x,y_0)-c^b(x)\\
&=&\sup_{x \in X}h(Q(x),y)-c^h(Q(x)), \text{ by (\ref{util})}\\
&=&\sup_{z \in Z}h(z,y)-c^h(z),\\
\end{eqnarray*}
as desired.  Turning to (ii), as $h(z,y)-c(y)= b(x,y) -b(x,y_0)-c(y)$ for $z=Q(x)$, then clearly $y_{c^b}(x) = y_{c^h}(Q(x))$.  Now $Q$ is continuous and surjective; in addition, it is straightforward to show $Q$ is an open mapping.  Therefore, continuity of  $y_{c^h}$ is equivalent to that of $y_{c^b}$. 
\end{proof}
\newtheorem{unique}[lin]{Corollary}

\begin{unique}\label{unique}
Assume $b$ satisfies $\textbf{(B0)}-\textbf{(B3)}$, that $c$ is $b$-convex and that the measure $\mu$ is absolutely continuous with respect to Lebesgue.  Then $P$ is concave.  If either (i) $b$ also satisifes  \textbf{(B3u)} or (ii) $c$ is strictly $b$-convex, then $P$ is strictly concave and its maximizer, the monopolist's optimal pricing plan, is unique.
\end{unique}
\begin{proof}
By the result in \cite{fkm}, we only need to verify that $\nu = Q_{\#}\mu$ is absolutely continuous with respect to $n$-dimensional Lebesgue measure.  This follows from the coarea formula (see \cite{p} for more details on the measure $\nu)$.
\end{proof}
Under these conditions, other economic phenomena such as bunching and the desirability of exclusion can be studied as in \cite{fkm}.  In particular, let us say a few words about bunching, or the phenomenon that sees different types choose the same good.  Of course, in this setting one naturally expects bunching because, as was noted by Basov \cite{bas}, the $m>n$ condition precludes the full separation of types.  When $X$ is $b$-convex, the bunching that occurs as a result of the difference in dimensions corresponds to identifying all types in a single level set $L_x$.  The results of this section imply that these are not genuinely different types; that is, that they can be treated as a single type $z$ without any loss of pertinent information.  However, genuine bunching occurs when types in different level sets opt for the same good;  Figalli, Kim and McCann conjecture this occurs under \textbf{(B3)} \cite{fkm}.
\newtheorem{rem}[lin]{Remark}
\begin{rem}
 In light of the previous section, these results mean that the monopolist's problem cannot be reduced to a maximization over a convex set when $m>n$ (at least as long as the extra dimensions encode real, economic information); this means that this class of problems is especially daunting.  However, in certain special cases these problems can be treated without relying on convexity.  Basov, for example, treats the case where $Y$ is a convex graph over $n$ variables embedded in $\mathbb{R}^m$ and $b(x,y)=x \cdot y$ \cite{bas}.  He then uses the techniques of Rochet and Chone \cite{rc} to solve the monopolist's problem in the epigraph (a convex, $n$-dimensional set) and shows that it is actually optimal to sell each consumer a product in the original graph.  The case $(m,n)=(2,1)$ with a general preference function is treated by Deneckere and Severinov \cite{ds}, again in the absence of a $b$-convex space of products.
\end{rem}

\section{$n>m$}

When $n>m$, the generalized Spence Mirlees condition cannot hold; that is, $y \mapsto D_xb(x,y)$ cannot be injective.  Therefore, when faced with a pricing schedule, a consumer's utility will typically be maximized by a continuum of products.  The principal has the ability to offer only the good which will maximize her profits from that consumer; however, in doing so, she may exclude products that maximize her profits from another consumer.  One way around this difficulty is to assume a tie-breaking rule as in Buttazzo and Carlier \cite{bc}; that is, assume that the principal can persuade each consumer to select the product that maximizes her profits (among those which maximize that consumer's utility function).  This is in fact inherent in Carlier's formulation of the problem and proof of existence \cite{car}. \footnote{In \cite{car}, no extended Spence Mirrlees condition is assumed and so the function $y_{v^b}(x) \in \argmax_{y \in \overline{Y}} b(x,y)-v(y)$ need not be uniquely determined by the $b$-convex $v^b(x)$.  If $v^b$ and $y_{v^b}$ maximize $P$, then $y_{v^b}$ must satisfy a tie-breaking rule;  that is, $y_{v^b}(x)$ must be chosen among elements in $\argmax_{y \in\overline{Y}} b(x,y)-v(y)$ so as to maximize the profits $v(y_{v^b}(x))-c(y_{v^b}(x))$.}

As we show in this section, this difficulty can be avoided by assuming $b$-convexity of $X$.  Much like in the last section, this condition will allow us to reduce to a problem where the dimensions of the two spaces are the same.  Intuitively, given a price schedule $v(y)$, a consumer $x$ will see the space of goods disintegrate into sub-manifolds.  If the price schedule is $b$-convex, then the consumer's preference $b(x,y)-v(y)$ for good $y$ will be maximized at every point in (at least) one of these submanifolds.  The $b$-convexity of $X$ will ensure that this disintegration will be the same for each $x \in X$.  The principal can then choose to offer only the good in each of these submanifolds which will maximize her profits from consumers whose utilities are maximized on that sub-manifold; the resulting space will be $m$ dimensional.  A special case of this structure was exploited by Basov to prove a similar result for bilinear preference functions \cite{bas}.

The motivation behind the $b$-convexity of $X$ is not as clear the motivation behind as the $b$-convexity of $Y$, which we saw in section 3 was necessary for the convexity of $U_b$.  It is, however, present in the work of Rochet and Chone \cite {rc} and Basov \cite{bas} on bilinear preference functions (where it reduces to ordinary convexity) as well as that of Figalli, Kim and McCann \cite{fkm}.  In the latter work, it is noted that $b$-convexity of $X$ implies ordinary convexity after a change of coordinates, which is essential in their proof of the genericity of exclusion modelled on the work of Armstrong \cite{arm}.

Using the method from the previous section, we note that if $X$ is $b$-convex, and the level sets 
\begin{equation*}
L_y(x):=\{\overline{y} \in X: D_xb(\overline{y},x)=D_xb(x,y)\}
\end{equation*}
 are all connected, then they are independent of $x$; in this case, we will denote them simply by $L_y$.  Define the effective space of profits, $W$, to be the image of $y \mapsto D_xb(x_0,y):=Q(y)$, for some fixed $x_0$.  The effective preference function, defined by $h(x,w)=b(x,y)-b(x_0,y)$, for $y \in Q^{-1}(w)$, is well defined.  We then define a new, effective cost function by
\begin{equation*}
g(w) = \inf_{\{y: Q(y)=w\}}c(y)-b(x_0, y).
\end{equation*}

Our aim is to show that the maximizing the monopolist's proftis is equivalent to maximizing her effective profits; that is, the profits she could earn by selling the set $W$ of effective goods with production costs $g$ to consumers with effective preference function $h$.  In practice, of course, she has access to the real goods $Y$.  However, offering a good $y$ that minimizes $c(y)-b(x_0, y)$ over $Q^{-1}(w)$ for some $w$, to buyers with preference function $b(x,y)$ is equivalent to offering the effective good $w$ to buyers with preference function $h(x,w)$.  The following result implies that, in order to maximize her profits, it is enough to offer \emph{only} goods $y$ with this property.  Denote by $M_w =\argmin_{y: Q(y)=w}c(y) -b(x_0,y)$.

\newtheorem{newC}{Proposition}[section]
\begin{newC}.  
 Given a price schedule $v(y)$ and corresponding utility $v^b(x)$.  Define $\tilde{v}(y)$ by 
 \begin{eqnarray}\label{best}
 \tilde{v}(y) &=& (v^b)^b(y) = \sup_{x \in \overline{X}}b(x,y)-v^b(x),  \text{ if $y \in M_w$ for some $w$.}\\  
 &=& \infty,  \text{ otherwise.} \nonumber
 \end{eqnarray}
   Then, $\tilde{v}^b(x) = v^b(x)$ and, for almost all $x$, and any $y \in \argmax_{y \in \overline{Y}} b(x,y)-v(y)$, and $\tilde{y} \in \argmax_{y \in \overline{Y}} b(x,y)-\tilde{v}(y)$ we have
   $v(y) - c(y) \leq \tilde{v}(\tilde{y}) - c(\tilde{y})$.  Furthermore, we have $\tilde{v}(\tilde{y}) - c(\tilde{y})=h(x,Q(\tilde{y}))-v^b(x) + g(Q(\tilde{y})) = b(x,\tilde{y})-v^{b}(x)+c(\tilde{y})$  
\end{newC}

The interpretation of this result is that, in order to maximize her profits, the monopolist should only offer those goods which maximize $y \mapsto b(x_0, y)-c(y)$ over the set $Q^{-1}(w)$ for some $w$.  Any utility function $v^b$ can be implemented by offering only these goods -- or charging $v(y) = \infty$ for all other goods $y$.  By doing this for a given utility, the monopolist forces each consumer to buy the good which offers her the highest possible profit.  This resulting profit is the same as her effective profit for that price schedule.

\begin{proof}
We first show $v^b(x) = \tilde{v}^b(x)$.   It is well known that $v^b(x) = \sup_{x \in \overline{X}}b(x,y) - (v^b)^b(y)$.  As $(v^b)^b(y) \leq \tilde{v}(y)$, we have $v^b(x) \geq \tilde{v}^b(x)$.    For a given $x$, choose $y$ such that $v^b(x) = b(x,y) -(v^b)^b(y)$.  The argument in the proof of Lemma \ref{con} implies that we have $v^b(x) = b(x,\overline{y}) -(v^b)^b(\overline{y})$ for all $\overline{y} \in L_y = Q^{-1}(w)$, where $w = Q(y)$. In particular, choosing $\tilde{y} \in M_w$, 
\begin{eqnarray*}
v^b(x)& =& b(x,\tilde{y}) -(v^b)^b(\tilde{y})\\
& =& b(x,\tilde{y}) -\tilde{v}(\tilde{y}) \\
&\leq& \sup_{y \in \overline{Y}}b(x,y) -\tilde{v}(y)\\
&=&\tilde{v}^b(x).
\end{eqnarray*}
Therefore, $v^b(x) = \tilde{v}^b(x)$.  Now, choose $x$ where $v^b$ is differentiable; this holds for almost all $x$.  For any $y \in \argmax_{y \in \overline{Y}} b(x,y)-v(y)$, we have $Dv^b(x) = D_xb(x,y)$.  Therefore $D\tilde{v}^b(x) = D_xb(x,y)$, which implies that, for any $\tilde{y} \in   \argmax_{\overline{y} \in \overline{Y}} b(x,\overline{y})-\tilde{v}(\overline{y})$ we must have $\tilde{y} \in L_y$.  Choosing $\tilde{y} \in M_w$ for $w = Q(y)$, we also have
\begin{eqnarray*}
v(y)-c(y) & = & b(x,y) -v^b(x) -c(y) \\
& \leq & \sup_{\overline{y} \in L_y}b(x,\overline{y}) -v^b(x) -c(\overline{y})\\
&=& \sup_{\overline{y} \in L_y} h(x,Q(y))+b(x_0,\overline{y})-v^b(x)-c(\overline{y})\\
&=& \sup_{\overline{y} \in L_y} h(x,Q(y))+b(x_0,\overline{y})-v^b(x)-c(\overline{y})\\
&=& h(x,Q(y)-v^b(x)+\sup_{\overline{y} \in L_y}b(x_0,\overline{y})-c(\overline{y})\\
&=& h(x,Q(y)-v^b(x)-\inf_{\overline{y} \in L_y}c(\overline{y})-b(x_0,\overline{y})\\
&=& h(x,Q(\tilde{y}))-v^b(x) -c(\tilde{y}) +b(x_0,\tilde{y})\\
&=& h(x,Q(\tilde{y}))-v^b(x) -g(Q(\tilde{y}))
\end{eqnarray*}
Noting that 
\begin{eqnarray*}
h(x,Q(\tilde{y}))-v^b(x) -c(\tilde{y}) +b(x_0,\tilde{y})& = & b(x, \tilde{y}) -v^b(x) -c(\tilde{y})\\
& =&(v^b)^b(\tilde{y})-c(\tilde{y})
\end{eqnarray*}
completes the proof.

\end{proof}

\newtheorem{reduce}[newC]{Corollary}
\begin{reduce}
Maximizing the principal's profits is equivalent to maximizing her effective profits.
\end{reduce}

\begin{proof}
The preceding result imlpies that, when maximizing $P$, it is sufficeient to only consider pricing plans of the form (\ref{best}), and that the profits from such a plan $\tilde{v}$ are the same as effective profits made from offering the effective pricing plan $(v^b)^h$.
\end{proof}

The uniqueness argument in \cite{fkm} relies on the $b$-convexity of $c$; we verify below that this convexity carries over when we reduce to an equal dimensional problem.

\newtheorem{b-con}[newC]{Proposition}
\begin{b-con}
(i)If $c$ is $b$-convex, $g$ is $h$-convex.\\
(ii) $c$ is strictly $b$-convex if and only if $g$ is strictly $h$-convex.
\end{b-con}
 \begin{proof}
  For a $b$ convex $c$, we have 
\begin{eqnarray*}
 g(w) & = & \inf_{\{y: D_xb(x_0,y)=w\}}c(y)-b(x_0,y) \\
& = & \inf_{\{y: D_xb(x_0,y)=w\}}\sup_{x \in X}{b(x,y)-c^b(x)}-b(x_0,y) \\
& = & \inf_{\{y: D_xb(x_0,y)=w\}}\sup_{x \in X}h(x,w)-c^b(x) \\
& = & \sup_{x \in X}h(x,w)-c^b(x) \\
\end{eqnarray*}

Now, if $c$ is strictly $b$-convex, then $y_{c^b}(x) = \argmax_{y \in \overline{Y}} b(x,y)-c(y)$ is continuous.  Therefore,

\begin{eqnarray*}
y_{c^b}(x) &=& \argmax_{y \in \overline{Y}} b(x,y)-c(y)\\
&=& \argmax_{y \in \overline{Y}} h(x,Q(y))-c(y) +b(x_0,y)
\end{eqnarray*}   
The maximization above can be broken into two steps: first, fix $w \in \overline{W}$ and choose $y \in Q^{-1}(w)$ to maximize  $-c(y) +b(x_0,y)$, or equivalently, minimze $c(y) -b(x_0,y)$.  Denote the resulting $y = F(w)$. The second step is to then choose $w$ to maximize $h(x,w) -c(F(w)) +b(x_0,F(w)) = h(x,w) -g(w)$; that is, to find $w_{g^h}(x)$.  We have shown, then, that $y_{c^b}(x)) = F(w_{g^h}(x))$.  Now note that, as $F(w) \in Q^{-1}(w)$, we have, $Q(y_{c^b}(x)) = w_{g^h}(x)$, which implies (ii), as in the proof of Proposition \ref{screencon}.

 \end{proof}
 
The converse of part (i) in the preceding result is false -- it is possible for $g$ to be $h$-convex without $c$ being $b$ convex. The reason for this is that the $b$ convexity of $c$ is a property that concerns every $y \in Y$, while the $h$ convexity of $g$ concerns only those $y$ that minimize $c(y) - b(x_0,y)$ over a level set $L_y$.  One can choose $c$ such that $c(y) = \sup_{x \in \overline{X}} b(x,y) -c^b(x)$ for any $y \in M_w$ for some $w$, but not for other goods $y \in \overline{Y}$.  Such a cost will not $b$-convex, but the corresponding $g$ will be $h$-convex.

As in the last section, under these conditions, rather than solving the principal-agent problem on $X \times Y$ with preference function $b$ and cost $c$ we can solve it on $X \times W$ with preference function $h$ and cost $g$.  In particular, direct analogues of Theorems \ref{eff}, \ref{prop} and \ref{b3} and Corollaries \ref{b3con} and \ref{unique} hold in this setting.
\section{Conclusions}
We have shown that the $b$-convexity of the the space of products $Y$ (one half of condition \textbf{(B2)}) is necessary for the set $U_{b,\phi}$ of allowable utilites to be convex.  Furthermore, when $m>n$, we have shown that under this condition, the problem reduces to a problem with equal dimensions.  A similar result holds when $n>m$ and $X$ is $b$-convex (the other half of\textbf{(B2)}).   

We used these observations to show that, nominally, the main result of Figalli, Kim and Mccann holds for \textit{any} $n$ and $m$: assuming \textbf{(B0)}-\textbf{(B2)}, $U_{b,\phi}$ is convex if and only if \textbf{(B3)} holds.  However, we should bear in mind that \textbf{(B2)} is a very strong condition when $n \neq m$; as mentioned above, it effectively reduces the problem to a new screening problem where both spaces have dimension $\min(n,m)$. 

Economic consequences can then be deduced as in \cite{fkm} under condition \textbf{(B3)}.

\end{document}